\documentclass[leqno,11pt,twoside]{article} 
\usepackage{amsmath, amsthm, amssymb}

\usepackage{geometry} 
\usepackage{titlesec}

\titleformat{\section}
  {\normalfont\scshape}{\thesection .}{1em}{}
\titleformat{\subsection}
  {\normalfont\scshape}{\thesubsection .}{1em}{}

\title{Geometric Koszul complexes, syzygies of K3 surfaces and the Tango bundle}
\date{}
\author{\scshape J\"urgen Rathmann}

\usepackage{mathrsfs}  
\usepackage{tikz-cd}
\usepackage{enumitem}
\setlist{nolistsep}

\newtheorem{theorem}{Theorem}[section]
\newtheorem*{ThA}{Voisin's Theorem}

\newtheorem{proposition}[theorem]{Proposition}

\newtheorem{lemma}[theorem]{Lemma}

\newtheorem{lemma1}{Lemma}[theorem]

\theoremstyle{remark}
\newtheorem{num}[theorem]{}
\newtheorem{remark}[theorem]{Remark}

\newtheorem{definition}[theorem]{Definition}

\newcommand\cE{\mathscr{E}}
\newcommand\cF{\mathscr{F}}
\newcommand\cG{\mathscr{G}}

\newcommand\cI{\mathscr{I}}
\newcommand\cK{\mathscr{K}}
\newcommand\cO{\mathscr{O}}

\newcommand\bP{\mathbf{P}}

\newcommand\tDelta{{\tilde{I}}}

\newcommand\sigmax{\sigma}

\DeclareMathOperator{\Sym}{Sym}

\DeclareMathOperator{\Ker}{Ker}
\DeclareMathOperator{\Img}{Im}

\DeclareMathOperator{\Grass}{Grass}

\DeclareMathOperator{\rk}{rk}
\DeclareMathOperator{\Hom}{Hom}

\DeclareMathOperator{\Hilb}{Hilb}
\DeclareMathOperator{\Bl}{Bl}

\def\hide#1{\relax}

\begin{document}
\bibliographystyle{plain}
\maketitle

{\small 
A key result for syzygies of curves is Voisin's proof of Green's conjecture 
for the canonical embedding of a general curve of any genus. Her primary tools were the 
Lazarsfeld Mukai bundle on a K3 surface and a representation of Koszul 
cohomology on the Hilbert scheme of points on the surface. 
In this note we construct representations of the Koszul complex 
on Grassmann varieties; Voisin's setup arises as the inverse image of
one of the maps.
Using a different map, we give a substantially shorter proof of Voisin's result 
for K3 surfaces of even sectional genus. 
}

\section{Introduction}

Let $X\subset\bP H^0(L)=\bP(V)=\bP^r$ be a projective variety, embedded by a very ample 
line bundle $L$.

We denote by $S=\Sym H^0(X,L)$ the homogeneous coordinate ring of $\bP^r$, by
$R=R(L)=\oplus_mH^0(X,mL)$
the graded $S$-module associated to $L$, and by $E_\bullet=E_\bullet(L)$ the minimal graded free resolution 
of $R$ over $S$:
\begin{equation*}
0\to E_{r-1}\to\dotsc\to E_2\to E_1\to E_0\to R\to 0.
\end{equation*}
Let $K_{p,q}(X;L)$ be the vector space of minimal generators of $E_p$ in degree $p+q$, so that
\begin{equation*}
E_p=\oplus_qK_{p,q}(X;L)\otimes S(-p-q).
\end{equation*}

A general theme in the study of syzygies is to relate them to special secant 
configurations on $X$. This suggests looking for ways to identify syzygies on 
Grassmann varieties as parameter spaces for linear subvarieties.

In this note we offer such an approach: For any given positive integer $k\le r$, we construct a complex 
of sheaves $\cK_\bullet^k$ on $\bP^r\times\Grass_k(V)$ with the property that for any coherent sheaf $\cF$
on $\bP^r$, the complex of sections of $pr_1^*\cF\otimes\cK_\bullet^k$  computes
the Koszul cohomology of $\cF$.

These complexes arise by splicing two subcomplexes:
\begin{enumerate}
\item a twist if the resolution of the ideal sheaf of the incidence variety in $\bP^r\times\Grass_k(V)$,
\item followed by the relative Koszul complex for the projection from the incidence variety 
to the Grassmannian.
\end{enumerate}

The direct image of this complex on $\bP^r$ recovers the Koszul complex,
while the projection to $\Grass_k(V)$ can use information on the geometry of $\cF$ along the 
$k$-planes in $\bP^r$.

Voisin \cite{V1} has constructed a representation of $K_{p,q}(X;L)$ on a certain 
open subscheme of the Hilbert scheme $\Hilb^p(X)$ of points on $X$. This 
representation is closely related to our construction, namely, it arises as
an inverse image under the canonical rational map $\Hilb^p(X)\to\Grass_p(V)$ which sends a subscheme of $p$
points to its linear span.

This approach has also been used by Ein and Lazarsfeld in their proof of the gonality conjecture \cite{EL}.

A famous application is Voisin's proof of Green's conjecture on the syzygies of the
canonical embedding of a general curve $C$ of even genus $2k$. Via taking a hyperplane section, 
it is based on the following result.

\begin{ThA}[\cite{V1}]
Let $(X,L)$ be a \textup{K3} surface whose Picard group is generated by a very ample
line bundle $L$ with $L\cdot L=4k-2$. Then we have $K_{k,1}(X;L)=0$.
\end{ThA}

Voisin (and later also Kemeny in his short proof \cite{K}) work with the space of sections of a Lazarsfeld Mukai
bundle, a certain rank 2 vector bundle on $X$ which is unique under the given 
conditions.\footnote{A completely different proof of Green's conjecture for the canonical embedding
of a general curve of any genus has recently been given by Aprodu, Farkas, Papadima, Raicu and Weyman \cite{AFP}.}

Voisin employs the inverse image of the representation of Koszul cohomology on 
the Hilbert scheme using an elaborate calculation.

Our application of the geometric Koszul complexes will also use a
Lazarsfeld Mukai bundle $E$ on $X$, but our approach differs in two points from Voisin's:
\begin{enumerate}
\item
We show the vanishing of the Koszul cohomology group $K_{k-2,2}(X;L)$ which is dual to 
$K_{k,1}(X;L)$. This is essentially a matter of taste, but perhaps a more natural choice
(compare Lazarsfeld's remark at the end of the introduction in \cite{EL}). This approach allows us
to simultaneously investigate the cases of K3 surfaces of even and odd sectional genus. 
\item
The zero sets of sections of $E$ consist of sets of $k+1$ points (counted with multiplicities) 
whose linear span has only dimension $k-1$. This means that the rational map from
$\Hilb^{k+1}(X)$ to $\Grass_{k+1}(V)$ is not defined on the images of the sections of $E$.
Voisin solves this problem by moving the image inside the Hilbert scheme.
In our approach, we work with the inverse image from a different Grassmann
variety, $\Grass_k(V)$, instead.
\end{enumerate}

Turning to a description of our proof, the starting point is the Koszul complex of sheaves on the 
K3 surface $X$ whose global sections compute Koszul cohomology. 

We employ three transformation steps along the following diagram from right to left:
\begin{equation*}
\begin{tikzcd}
X\times\bP(H^0(E)^\vee) \ar[r] \ar[d] & X\times\Grass_k(H^0(L)) \ar[d] \\
\bP(H^0(E)^\vee) & X
\end{tikzcd}
\end{equation*}
\begin{enumerate}
\item The Koszul complex on $X$ can be represented as the direct image of a complex on the product of $X$ with
any (fixed) Grassmannian. We construct these complexes in section \ref{sec2}, and 
establish their properties. As the usual Koszul complex arises as a direct image, sections (and
cohomology groups of the complex of sections) are the same.
\item The K3 surface $X$ carries a unique (for even sectional genus) Lazarsfeld Mukai bundle $E$, 
and its space of sections maps to the
Grassmannian. The relevant part of the geometric Koszul complex is supported on the incidence 
variety in $X\times\Grass_k(V)$ and consists of locally free sheaves. 
Hence it is locally split and the inverse image on $X\times\bP(H^0(E)^\vee)$ remains exact.
\item Finally, we project the complex to the parameter space for the sections.
As a direct image, the projection map does not change the sections. We find that the inverse images of the sections from 
$X\times\Grass_k(V)$ are actually supported on a subcomplex of the direct image complex (see \ref{lem42}).
The Koszul cohomology group is identified as the first cohomology of a line bundle on projective space, 
and hence vanishes.
\end{enumerate}

Our proof crucially depends on properties of the inverse image of the universal quotient bundle on
the Grassmannian. Its pullback to $\bP=\bP(H^0(E)^\vee)$ turns out to be well known in a different context. 
It was first constructed by
Tango \cite{T} as an example of a rank $k$ vector bundle on $\bP^{k+1}$.

Its key property is its minimal graded resolution which has the particularly simple form (compare \cite[3.47]{V2})
\begin{equation*}
0\to \cO(-2)\to H^0(E)\otimes\cO(-1)\to H^0(L)\otimes\cO\to Q'\to 0.
\end{equation*}

As Voisin's theorem does not hold unconditionally over base fields of finite characteristic, 
our proof needs to use appropriate tools. The dependence on the characteristic manifests itself
in two points: 
\begin{enumerate}
\item[(i)]
The resolutions for $\wedge^i Q'$ $(i\le k)$ in Theorem \ref{thm44} require that the characteristic of 
the base field exceeds $i$, hence the characteristic must be $0$ or larger than $k$.
\item[(ii)]
As mentioned above, the transformed sections
live in a subcomplex of the transformed Koszul complex on $\bP$. This subcomplex will be exact
if and only if the (degree $k+1$) trace map from the incidence variety $I_E\subset X\times\bP\to\bP$ is an isomorphism.
This holds only if the characteristic of the base field does 
not divide $k+1$ (see Lemma \ref{lem432}).
\end{enumerate}

Our proof does not extend to K3 surfaces with odd sectional genus, because the inverse image homomorphism is no longer
injective on sections. For details and further discussion, see (\ref{rem46}) and (\ref{rem47}).

Finally, we would like to point out that the application to the space of sections of the 
Lazarsfeld-Mukai bundle may appear counter-intuitive: The geometric representations 
of Koszul cohomology are geared to detect secant spaces to a variety, or to encode the 
absence of such secant spaces (e.g. used in the proof of the gonality conjecture \cite{EL}). 
In our case, all the sections correspond to secant spaces.

I am grateful to Michael Kemeny and Rob Lazarsfeld for correspondence.
The material in section 2 is considerably older than the application to K3 surfaces. I have benefited from
discussions with David Eisenbud, Rob Lazarsfeld and Ruijie Yang.

\section{Syzygies via the Grassmannian}\label{sec2}

This section constructs the geometric Koszul complexes and establishes their key properties. 
Our proof in section \ref{sec4} only
uses the fact that the short exact sequence in (\ref{sect29}.1) calculates Koszul cohomology. 
In comparison, Voisin \cite{V1} uses the sequence in (\ref{sect29}.2), whereas Kemeny's proof 
\cite{K} takes place in the setting of (\ref{sect29}.3), but without using the map to the Grassmannian.

\begin{num}
In this section, $X\subset\bP H^0(X,L)=\bP^r$  is a smooth projective manifold embedded by the 
sections of a very ample line bundle $L$,
$\cF$ is a coherent sheaf on $X$, not necessarily locally free.

The shape of the minimal graded resolution of the module $\oplus_l H^0(X,\cF(l))$ over the graded ring
$S=\Sym H^0(X,L)$ depends on the Koszul cohomology groups $K_{p,q}(X,\cF;L)$.

These can be calculated from the standard Koszul complex of sheaves 
\begin{equation*}
\cK_\bullet=\dotsc\to \wedge^2V\otimes\cO(-2)\to V\otimes\cO(-1) \to\cO\to 0
\end{equation*}
on $\bP^r$ by twisting
with $\cF(p+q)$, and taking the $p$-th homology group of the complex of global sections:
i.e., 
\begin{equation*}
K_{p,q}(X,\cF;L)=H_p\Big(\Gamma\big(X,\cK_\bullet\otimes\cF(p+q)\big)\Big).
\end{equation*} 
\end{num}

\begin{num}
Given two sheaves $\cF$, $\cG$ on a projective manifold $X$, the multiplication map of sections
\begin{equation*}
H^0(\cF)\otimes H^0(\cG)\to H^0(\cF\otimes\cG)
\end{equation*}
can be studied with geometric tools on $X\times X$ via the map of sheaves
\begin{equation*}
pr_1^*\cF\otimes pr_2^*\cG \to (pr_1^*\cF\otimes pr_2^*\cG)\otimes\cO_\Delta
\end{equation*}
where $\Delta\subset X\times X$ is the diagonal.

The groups $\wedge^i H^0(X,L)$ occurring in the Koszul complex are naturally represented
as the sections of the tautological positive line bundle on the Grassmannian $\Grass_i(V)$. This suggests
to look for a representation of the Koszul map on $X\times\Grass_i(V)$, where the incidence variety $I_{i,r}$
takes the role of the diagonal, namely as
\begin{equation*}
pr_1^*\cF\otimes pr_2^*\cO_{\Grass}(1) \to \big(pr_1^*\cF\otimes pr_2^*\cO_{\Grass}(1)\big)\otimes\cO_{I_{i,r}}.
\end{equation*}
It is surprising that not only this works, but that the whole Koszul complex can be represented.
\end{num}

\begin{num}
Let $\Grass_i(V)$ be the Grassmann variety of $i$-dimensional quotients of a
vector space $V$ of dimension $r+1$ (corresponding to linear subspaces of $\bP^r$ of dimension $i-1$),
with universal subbundle $S$ and universal quotient bundle $Q$.

The incidence variety $I_{i,r}=\{(x,L)\,\vert\,x\in H\}\subset\bP^r\times\Grass_i(V)$ is the 
vanishing scheme of the composition $pr_2^*S\to V\otimes\cO\to pr_1^*\cO(1)$,
hence there is an exact sequence
\begin{equation}\label{eq1}
0\to \wedge^{r+1-i} \tilde S\to \dotsc\to \tilde S\to \cO\to\cO_I\to 0
\end{equation}
where $\tilde S=pr_1^*\cO(-1)\times pr_2^*S$.

Denoting by $\pi_1,\pi_2$ the restrictions of the two projections to $I_{i,r}$, the map $\pi_2^*S\to\pi_1^*\cO(1)$
vanishes, thus there is a surjection $\pi_2^*Q\to\pi_1^*\cO(1)$ on $I_{i,r}$ and an exact sequence
\begin{equation}\label{eq2}
0\to \wedge^i \tilde Q\to\dotsc \to  \wedge^2\tilde Q\to \tilde Q\to \cO_I\to 0
\end{equation}
where $\tilde Q=\pi_1^*\cO(-1)\otimes\pi_2^* Q$. 

The morphism $\pi_2$ expresses $I_{i,r}$ as a relative $\bP^i$-bundle over $\Grass_i(V)$; 
and the complex (\ref{eq2}) is the corresponding relative Koszul complex.
\end{num}

\begin{definition}
The \emph{geometric Koszul complex $\cK^i_\bullet$} on $\bP^r\times\Grass_i(V)$ is the join of
the following two complexes
\begin{enumerate}
\item[(i)] the twist of (\ref{eq1}) with $\det(\tilde S)^{-1}=pr_1^*\cO(-i)\otimes pr_2^*\det(S)^{-1}$, and
\item[(ii)] the complex (\ref{eq2}).
\end{enumerate}
\end{definition}

The reader will note that the complex (i) ends with $\det(\tilde S)^{-1}\otimes\cO_I=\pi_1^*\cO(-i)\otimes \pi_2^* \det(Q)$,
and this is the first term of the complex (ii).

Key properties of the complexes $\cK^i_\bullet$ are as follows: 

\begin{theorem}\label{thm25}
Let $\cK^i_\bullet$ be the geometric Koszul complex on $\bP^r\times\Grass_i(V)$.
\begin{enumerate}
\item[\textup{1.}] $pr_{1,*}\cK^i_\bullet$ is isomorphic to the Koszul complex $\cK_\bullet$ on $\bP^r$.
\item[\textup{2.}] Let $\cF$ be a coherent sheaf on $\bP^r$. Then we have
\begin{equation*}
pr_{1,*}(pr_1^*\cF\otimes \cK^i_\bullet)\cong \cF\otimes  pr_{1,*}\cK^i_\bullet= \cF\otimes\cK_\bullet.
\end{equation*}
\end{enumerate}
\end{theorem}
Therefore $\cK^i_\bullet$ can be used to calculate the Koszul cohomology groups of $\cF$.
\begin{proof}
1. Regarding the sheaves in the complexes, Bott's formula \cite[4.1.12 and 4.1.8]{W} shows 
(after taking some effort to identify the sheaves involved) that
\begin{equation*}
R^lpr_{1,*}\big(\cK^i_j\big)=
\begin{cases}
\cK_j  & \text{ for } l=0 \\
0 & \text{ for } l>0,
\end{cases}
\end{equation*}
hence the spectral sequence for the hyperdirect image of $pr_1$ implies that $pr_{1,*}(\cK^i_\bullet)$ 
has the same terms as $\cK_\bullet$ and is exact.

It remains to be shown that the maps of the two complexes correspond to each other,
which requires tracking the maps during the proof of Bott's formula.
This is straightforward when working on an appropriate flag variety, 
but not very illuminating for the reader. We omit the details.

2.
Given an arbitrary map $f\colon X\to Y$ and a coherent $\cO_X$-module $\cE$, the adjointness of $f_*$ and $f^*$ 
provides a natural map $f^*f_*\cE\to\cE$. For any $\cO_Y$-module $\cF$, we have natural bijections
\begin{multline*}
\Hom\big(f^*f_*\cE\otimes f^*\cF,\cE\otimes f^*\cF\big) \to\Hom\big(f^*(f_*\cE\otimes\cF),\cE\otimes f^*\cF\big) \\
\to \Hom\big(f_*\cE\otimes\cF,f_*(\cE\otimes f^*\cF)\big).
\end{multline*}
Corresponding to the adjunction map $f^*f_*\cE\to\cE$, the induced map
\begin{equation*}
f^*f_*\cE\otimes f^*\cF\to\cE\otimes f^*\cF
\end{equation*}
leads to a natural map
\begin{equation*}
f_*\cE\otimes\cF\to f_*(\cE\otimes f^*\cF).
\end{equation*}
If $\cF$ is locally free, this map is well known to be an isomorphism (\emph{projection formula}).

In our situation, we obtain a natural map of complexes
\begin{equation*}
\cK_\bullet\otimes\cF\cong pr_{1,*}\big(\cK^i_\bullet\big)\otimes\cF\to pr_{1,*}\big(\cK^i_\bullet\otimes pr_1^*\cF\big),
\end{equation*}
hence we only need to show that the maps of sheaves
\begin{equation*}
pr_{1,*}(\cK^i_j)\otimes\cF\to pr_{1,*}\big(\cK^i_j\otimes pr_1^*\cF\big)
\end{equation*}
are isomorphisms for every $j$.

Starting from a resolution 
\begin{equation*}
\dotsc\to\cF_2\to\cF_1\to\cF_0\to\cF\to 0
\end{equation*}
of $\cF$ on $\bP^r$ by locally free sheaves, there is a commutative diagram
\begin{equation*}
\begin{tikzcd}[column sep=small]
f_*\cE\otimes\cF_1 \ar[r]\ar[d] & f_*\cE\otimes\cF_0 \ar[r]\ar[d] & f_*\cE\otimes\cF \ar[r]\ar[d] & 0 \\
f_*(\cE\otimes f^*\cF_1) \ar[r] &f_*(\cE\otimes f_*\cF_0) \ar[r] &f_*(\cE\otimes f^*\cF) \ar[r] & 0
\end{tikzcd}
\end{equation*}
where $\cE$ is one of the $\cK^i_j$ and $f$ is either 
$pr_1\colon \bP^r\times\Grass_i(V)\to\bP^r$ or $\pi_1\colon I_{i,r}\to\bP^r$,
depending on whether $j>i+1$ or $j\le i+1$.

The two vertical maps on the left and in the middle are isomorphisms by the projection
formula for locally free sheaves, hence the vertical map on the right 
will also be bijective, if the two rows are exact.

The top row is certainly exact, since $f_*\cE$ is locally free, hence flat.

Regarding the bottom row, note that $pr_1$ (resp.\ $\pi_1$) is flat and $\cE$ 
is locally free on $\bP^r\times\Grass_i(V)$ (resp.\ $I_{i,r}$), hence the complex 
\begin{equation*}
\dotsc\to \cE\otimes f^*\cF_2\to \cE\otimes f^*\cF_1\to \cE\otimes f^*\cF_0\to \cE\otimes f^*\cF\to 0
\end{equation*}
remains exact.

Applying $f_*$, we obtain the complex
\begin{equation*}
\dotsc\to f_*(\cE\otimes f^*\cF_2)\to f_*(\cE\otimes f^*\cF_1)\to f_*(\cE\otimes f^*\cF_0)\to f_*(\cE\otimes f^*\cF)\to 0
\end{equation*}
whose exactness can be investigated using the spectral sequence
for the hyperdirect image of $f$.
Exactness at $f_*(\cE\otimes f^*\cF)$ requires the vanishing of 
$R^lf_*(\cE\otimes f^*\cF_l)$ for $l\ge 1$, while exactness at $f_*(\cE\otimes f^*\cF_0)$ requires 
the vanishing of $R^lf_*(\cE\otimes f^*\cF_{l+1})$ for $l\ge 1$.

Now the projection formula for locally free $\cF_j$ shows that
\begin{equation*}
R^lf_*(\cE\otimes f^*\cF_j)\cong R^lf_*(\cE)\otimes \cF_j,
\end{equation*}
which reduces the question to the well-known vanishing of $R^lf_*(\cE)$ for $l\ge 1$ (Bott's formula).
\end{proof}

\begin{remark}\label{num26}
Our proof shows that for $k\le i-1$ the restriction of sections
\begin{multline*}
H^0(\bP^r\times\Grass_i(V),pr_1^*\cF\otimes pr_2^*\wedge^kQ)=
H^0(\bP^r,\cF)\otimes H^0(\Grass_i(V),\wedge^kQ) \\
\to H^0(I_{i,r},\pi_1^*\cF\otimes\pi_2^*\wedge^k Q)
\end{multline*}
is bijective for every coherent sheaf $\cF$ on $\bP^r$. This will be used in the proof of
(\ref{lem42}) below.
\end{remark}

\begin{num}
For any morphism $f\colon Y\to X\times\Grass_i(V)$, there is an inverse image homomorphism
\begin{equation*}
K_{p,q}(X,\cF;L)\to H_p\Big(\Gamma\big(Y,f^*(\cK^i_\bullet\otimes pr_1^*\cF(p+q))\big)\Big).
\end{equation*}

This map is derived from standard constructions on sheaves, as follows:
Given a morphism $f\colon X_1\to X_2$, the adjointness
of $f^*$ and $f_*$ defines a natural
transformation $\text{Id}\to f_*f^*$. As $\Gamma(X_2,f_*\cG)=\Gamma(X_1,\cG)$
for any sheaf $\cG$ on $X_1$,
one obtains a natural map of global sections
\begin{equation*}
\Gamma(X_2,\cG')\to\Gamma(X_2,f_*f^*\cG')=\Gamma(X_1,f^*\cG')
\end{equation*}
for any quasi-coherent sheaf $\cG'$ on $X_2$. This construction extends to complexes of sheaves,
and to their homology.
\end{num}

\begin{num}
Base change can in particular be applied to spaces of the form $X\times Z$, 
where $Z$ is a parameter space of points on $X$, and where the rational map
to the Grassmannian maps a set of points to its linear span:
\begin{enumerate}
\item[(i)] the $i$-th symmetric product $X^{(i)}$ of $X$,
\item[(ii)] the $i$-th cartesian product $X^i$ of $X$ \cite{ELY},
\item[(iii)] the principal component $\Hilb_{princ}^i(X)$ of the Hilbert scheme of $0$-dimensional 
subschemes of $X$ of fixed length $i$ \cite{V1}.
\end{enumerate}
We have the following rational maps (dashed arrows) and morphisms (solid arrows) which are
compatible with the corresponding inverse images:
\begin{equation*}
\begin{tikzcd}
X^i \ar[r, dashed] \ar[dr] & \Hilb^i_{princ}(X) \ar[r, dashed] \ar[d] & \Grass_i(V) \\
& X^{(i)} \ar[ru,dashed]
\end{tikzcd}
\end{equation*}
The two maps to $\Grass_i(V)$ send a set of $i$ points to its span, and the vertical map in the
middle is the (birational) Hilbert-Chow morphism.

If the embedding $X\to\bP^r$ is $i$-very ample (for the definition, see e.g.~\cite{EL}), then the map 
$\Hilb^i_{princ}(X)\to\Grass_i(V)$ is defined everywhere. This indicates that the principal
component of the Hilbert scheme is the canonical choice for base change.

A key criterion is to determine to what degree the inverse image map of the Koszul complex
is bijective on sections. 
\end{num}

\begin{num}
In general there are two challenges during the implementation:
\begin{enumerate}
\item
The morphism to the Grassmannian is only defined on points of the parameter space
which correspond to points of $X$ in general position, hence may exist only as a rational map. 
This can be addressed by restricting the map to the subscheme of points in general position, or 
by blowing up the locus of indeterminacy.\footnote{On the principal component of the Hilbert
scheme of points, the second part of the complex $\cK^i_\bullet$, represented by
the sequence (\ref{eq2}), can always be reconstructed, using the surjection 
$E_L=\pi_2^*\pi_{2,*}\pi_1^*\to\pi_1^*L$,
even when the embedding $X\to\bP^r$ is not $i$-very ample. Here $\pi_1$ resp.\ $\pi_2$ are the
projections from the incidence variety to $X$ resp.\ $\Hilb^i_{princ}(X)$.}
\item
The inverse image functor preserves exactness on the right part of $\cK_\bullet^i$ (corresponding to the 
complex (\ref{eq2}) which is supported on
the incidence variety), but not on the left part (corresponding to the sequence (\ref{eq1})).
The standard remedy is to
blow up the incidence variety and to split the complex (\ref{eq1}) into short exact sequences before
taking inverse images.
\end{enumerate}
\end{num}

\begin{num}\label{sect29}
We now describe in more detail the short exact sequences of sheaves, whose corresponding
sequences of global sections calculate Koszul cohomology:
\begin{enumerate}
\item[1.] $p<i-1$: In this case, one works on the incidence variety $I_{i,r}$. The key diagram of vector bundles is as follows:
\begin{equation}\label{eq3}
\begin{tikzcd}
&& 0 \ar[d] & 0 \ar[d] \\
0 \ar[r] & \pi_2^*S \ar[r] \ar[d, equals] & \pi_1^*\Omega^1 \ar[r] \ar[d] & R \ar[r] \ar[d] & 0 \\
0 \ar[r] & \pi_2^*S \ar[r] & V\otimes\cO \ar[r] \ar[d] & \pi_2^*Q \ar[r] \ar[d] & 0 \\
&& \pi_1^*\cO(1) \ar[r,equals] \ar[d] & \pi_1^*\cO(1) \ar[d] \\
&& 0 & 0
\end{tikzcd}
\end{equation}
The short exact sequence of sheaves (calculating $K_{p,q}(\cF)$) is determined by tensoring
\begin{equation*}
0\to \wedge^{p+1} R \otimes\pi_1^*\cO(q-1) \to \wedge^{p+1} \pi_2^*Q \otimes\pi_1^*\cO(q-1) 
\to \wedge^p R \otimes\pi_1^*\cO(q)\to 0
\end{equation*}
with $pr_1^*\cF$.

\item[2.] $p=i-1$: This is a special case: $K_{p,q}(X,\cF;L)$ is calculated from
\begin{equation*}
0\to \cI_{I_{i,r}}\to\cO \to \cO_{I_{i,r}}\to 0
\end{equation*}
by tensoring with $pr_1^*\cF(q-1)\otimes pr_2^*\cO_{\Grass}(1)$ and taking global sections. 

Voisin \cite{V1} has shown that the pullback of this sequence to the subscheme of 
curvilinear points of the Hilbert scheme of points calculates Koszul cohomology.

\item[3.] $p>i-1$:
We let 
\begin{equation*}
\phi\colon B=\Bl_{I_{i,r}}(X\times\Grass_i)\to X\times\Grass_i
\end{equation*}
be the blow up, $\phi_1=pr_1\circ\phi$, $\phi_2=pr_2\circ\phi$, $\tDelta$ the preimage of $I_{i,r}$.

The key diagram in this case is as follows:
\begin{equation}\label{eq4}
\begin{tikzcd}
& 0 \ar[d] & 0 \ar[d] & 0 \ar[d] \\
0 \ar[r] & \bar S \ar[r] \ar[d] & \phi_1^*\Omega^1 \ar[r] \ar[d] & \bar Q \ar[r] \ar[d] & 0 \\
0 \ar[r] & \phi_2^*S \ar[r] \ar[d] & V\otimes\cO \ar[r] \ar[d] & \phi_2^*Q \ar[r] \ar[d] & 0 \\
0 \ar[r] & \phi_1^*\cO(1)\otimes\cO(-\tDelta) \ar[r] \ar[d] & \phi_1^*\cO(1) \ar[r] \ar[d] 
& \phi_1^*\cO(1)\otimes\cO_\tDelta \ar[r] \ar[d] & 0 \\
& 0 & 0 & 0
\end{tikzcd}
\end{equation}

The sheaves $\bar S$ and $\bar Q$ are defined as the kernels of the vertical maps. We note that
all sheaves are locally free on $B$ except the one in the bottom right.

The Koszul cohomology group $K_{p,q}(X,\cF;L)$ is calculated from the short exact sequence
\begin{equation*}
0\to\wedge^{p+1-i}\bar S \to \wedge^{p+1-i}\phi_2^*S \to\wedge^{p-i}\bar S\otimes\phi_1^*\cO(1)\otimes\cO(-\tDelta)   \to 0
\end{equation*}
by twisting with $\phi_1^*\cF(q-1)\otimes\phi_2^*\cO_{\Grass}(1)$ and taking global sections.

\end{enumerate}
\end{num}

\begin{num}
Kemeny's proof of Voisin's theorem in \cite{K} explicitly constructs the inverse 
image of diagram (\ref{eq4}) (see the diagram at the beginning of section 1 in his paper) 
without recourse to the map to the Grassmannian.
\end{num}

\section{Geometry of the Lazarsfeld Mukai bundle}

This section summarizes the relevant properties of the Lazarsfeld Mukai bundle and its space of sections.
For the original construction, see \cite{L0}.
There is considerable overlap with \cite{V1} and \cite{K}. 

\begin{num}
Let $(X,L)$ be a polarized K3 surface, $E$ the rank two Lazarsfeld Mukai bundle  \cite{L0} associated to a 
basepointfree $g^1_{k+1}$ on a smooth curve
$C\in\vert L\vert$; there is an exact sequence
\begin{equation*}
0\to E^\vee\to H^0(C,A)\otimes\cO_X\to A\to 0
\end{equation*}
with $A$ a $g^1_{k+1}$ on $C$. We further have
\begin{enumerate}
\item $e=_{\text{def}}h^0(E)=h^0(A)+h^1(A)$, $h^1(E)=h^2(E)=0$;
\item $\det(E)=L$, $c_2(E)=k+1$.
\end{enumerate}

For every section $s$ of $E$ with $0$-dimensional vanishing scheme $\xi(s)$ there is a short exact sequence
\begin{equation*}
0\to \cO_X\to E\to \cI_\xi\otimes L \to 0
\end{equation*}
which shows that $h^1\big(X,\cI_\xi\otimes \cO_X(C)\big)=1$, i.e., the $k+1$
points of $\xi$ are not in general position in the embedding of $X$ by $\vert C\vert$, but span 
a linear space of dimension $k-1$.
\end{num}

\begin{num}
If $E$ is globally generated, we have the following diagram on 
$X\times\bP$ where $\bP=\bP(H^0(E)^\vee)$:
\begin{equation*}
\begin{tikzcd}
&& 0 \ar[d] \\
&& pr_2^*\Omega^1_\bP(1) \ar[d] \ar[dr, dashed] \\
0 \ar[r] & pr_1^*E^\vee \ar[r] \ar[dr, dashed] & H^0(E)^\vee\otimes\cO\ar[r] \ar[d] & pr_1^*F \ar[r] & 0 \\
&& pr_2^*\cO(1) \ar[d] \\
&& 0
\end{tikzcd}
\end{equation*}

The zero scheme $I_E\subset X\times\bP$ of the universal section of $E$ is the subscheme where the composition
\begin{equation*}
pr_1^*E^\vee\to H^0(E)^\vee\otimes\cO\to pr_2^*\cO(1)
\end{equation*}
drops rank, 
hence is resolved by an exact sequence
\begin{equation}\label{eq5}
0\to \wedge^2 E^\vee\boxtimes\cO(-2)\to E^\vee\boxtimes \cO(-1)\to \cO_{X\times\bP}\to \cO_{I_E}\to 0.
\end{equation}

We note that $I_E$ is isomorphic to $\bP(F)$, where
$F=M_E^\vee$ is the Lazarsfeld Mukai bundle corresponding to $\omega_C\otimes A^{-1}$.
\end{num}

\begin{num}\label{num33}
After twisting (\ref{eq5}) by $pr_1^*L=pr_1^*(\wedge^2E)$, the hypercohomology spectral sequence
with respect to $pr_{2,*}$ degenerates into the exact sequence
\begin{multline}\label{eq6}
0\to \cO_\bP(-2)\to H^0(E)\otimes\cO_\bP(-1)\to H^0(L)\otimes\cO_\bP\to \\
pr_{2,*}(\cO_{I_E}\otimes pr_1^*L)\to\cO_\bP(-2)\to 0
\end{multline}
and we can read off:
\begin{enumerate}
\item[(i)] The sheaf $S'=pr_{2,*}(\cI_{I_E}\otimes pr_1^*L)\cong\wedge^{e-2}\Omega_\bP^{1}(1)$ is $1$-regular in the
sense of Castel\-nuovo-Mumford, and $R^1 pr_{2,*}(\cI_{I_E}\otimes pr_1^*L)=\cO(-2)$.
In particular, we note that $\det(S')=\cO_\bP\big(-(e-2)\big)$ (\cite[Lemma 2]{V1}, \cite[Lemma 2.2]{K}).
\item[(ii)] The sheaf $Q'=\Ker\big(pr_{2,*}(\cO_{I_E}\otimes pr_1^*L)\to\cO_\bP(-2)\big)$ is $0$-regular,
hence the last map in (\ref{eq6}) splits, i.e.
\begin{equation*}
pr_{2,*}(\cO_{I_E}\otimes pr_1^*L)\cong Q'\oplus\cO_\bP(-2).
\end{equation*}
In certain cases (see (iv) below), we can identify a canonical splitting.
\end{enumerate}
\end{num}

\begin{num}\label{num34}
If all the sections of $E$ vanish in codimension $2$, there is a morphism
\begin{equation*}
f\colon\bP=\bP(H^0(E)^\vee)\to\Grass_k(V),
\end{equation*}
and we further note:
\begin{enumerate}
\item[(iii)] The universal subbundle on $\Grass_k(V)$ pulls back to $f^*S=S'$, and the 
universal quotient bundle on $\Grass_k(V)$ pulls back to $f^*Q=Q'$.

The restriction $\wedge^i Q\to \wedge^i Q'=\wedge^i Q\otimes\cO_\bP$
induces for each $i$ a map on sections
\begin{equation*}
\Phi_i\colon H^0(\Grass_k(V),\wedge^i Q)=\wedge^i H^0(L)\to H^0(\bP,\wedge^i Q').
\end{equation*}
In  (\ref{thm44}) below we provide criteria to ensure that $\Phi_i$ is bijective.
\item[(iv)]
Applying $pr_{2,*}$ to the sequence (\ref{eq5}), we obtain the exact sequence
{\small
\begin{equation*}
\begin{tikzcd}[column sep=tiny, row sep=tiny]
0\ar[r] & \cO_\bP\ar[r] & pr_{2,*}(\cO_{I_E})\ar[r] & R^2 pr_{2,*}L^\vee(-2)\ar[r] \ar[d, equals] & 
R^2 pr_{2,*} E^\vee(-1)\ar[r] \ar[d, equals] & R^2 pr_{2,*}(\cO_{X\times\bP})\ar[r] \ar[d, equals]& 0 \\
&&& H^0(L)^\vee \otimes\cO(-2) & H^0(E)^\vee \otimes\cO(-1) & \cO_\bP
\end{tikzcd}
\end{equation*}
}
\!from which we can read off (e.g. using relative duality for the projection $X\times\bP\to\bP$) that
\begin{equation*}
pr_{2,*}(\cO_{I_E})\cong pr_{2,*}(\cO_{I_E}\otimes pr_1^*L)^\vee\otimes\cO_\bP(-2)\cong Q'^\vee(-2)\oplus\cO.
\end{equation*}
The trace map $pr_{2,*}\cO_{I_E}\to\cO_\bP$ yields a natural splitting, provided it is non-zero.
The latter condition is equivalent to the characteristic of the base field not dividing $k+1$.
\item[(v)] We can embed the bundle $Q'$ into a diagram
\begin{equation*}
\begin{tikzcd}
&& 0 \ar[d] & 0 \ar[d] \\
&& K\otimes\cO \ar[r, equals] \ar[d] & K\otimes\cO \ar[d] \\
0 \ar[r] & S' \ar[r] \ar[d, equals] & \wedge^2 H^0(E)\otimes\cO \ar[r] \ar[d] & \wedge^2 S'\otimes\cO(2) \ar[r] \ar[d] & 0 \\
0 \ar[r] & S' \ar[r] & H^0(L)\otimes\cO \ar[r] \ar[d] & Q' \ar[r] \ar[d] & 0 \\
&& 0 & 0
\end{tikzcd}
\end{equation*}
$Q'$ is a rank $k$ vector bundle whereas $\bP$ has dimension $h^0(A)+h^1(A)-1=g-k+1$, and this 
requires that the Chern classes $c_{j}\big(\wedge^2 S'\otimes\cO(2)\big)$ vanish for $j\ge k+1$.
Hence this only works for $k\ge g-k$ \cite{T}, i.e., $2k\ge g$. In particular, we recover 
the well-known fact that any globally generated rank $2$ Lazarsfeld Mukai bundle with $2k<g$ must 
have sections vanishing on a curve.

For $r=2k$ we find that $Q'$ is a rank $k$ bundle on $\bP^{k+1}$.
These special bundles have been constructed first by Tango \cite{T}.
\end{enumerate}
\end{num}

\section{Base change}\label{sec4}

In this section we compare the sections of the geometric Koszul complex on the 
incidence variety over the Grassmannian with the sections of its inverse image
on the universal section of the Lazarsfeld Mukai bundle $I_E\subset X\times\bP$. 

We also give a direct proof of Voisin's theorem.

\begin{num}
We focus on the following situation: $X$ is a K3 surface, $L$ a very ample line bundle 
on $X$ such that every effective divisor in the associated linear system is reduced and irreducible.
We let $L\cdot L=4k-2\sigmax-2$ with $\sigmax\in\{0,1\}$. 
Later on we will specialize to
$\sigmax=0$, but for the moment we will track the parity of the sectional genus:
\begin{enumerate}
\item[(i)]
$\vert L\vert$ embeds $X$ into $\bP^r$ with $r=2k-\sigmax$, and $g(C)=2k-\sigmax$
for a general hyperplane section $C$. 
\item[(ii)]
The minimal degree of a pencil on $C$ has degree $k+1$, and we consider
the rank $2$ Lazarsfeld Mukai bundle $E$ corresponding to such a pencil.
Under our assumptions, $E$ is globally generated, and all the sections of $E$ vanish 
in codimension $2$. 
\item[(iii)]
We obtain a morphism $f\colon \bP(H^0(E)^\vee)\to\Grass_k(H^0(L))$; 
we note that $\rk f^*Q =k$, and $\rk f^*S =\dim \bP=k+1-\sigmax$, where
$S$ and $Q$ are the universal sub- and quotient bundle on the Grassmannian.
\end{enumerate}
\end{num}

In section \ref{sec2} we showed that the Koszul complex for $(X,L)$ on $X$ is the direct image of a complex
\begin{equation*}
\dotsc\to\wedge^{i+1} \pi_2^*Q\otimes \pi_1^*L^{\otimes j-1}\to \wedge^{i} \pi_2^*Q\otimes\pi_1^*L^{\otimes j}\to
\wedge^{i-1} \pi_2^*Q\otimes\pi_1^*L^{\otimes j+1}\to\dotsc
\end{equation*}
on the incidence variety in $X\times\Grass_k(H^0(L))$. Here $Q$ is the universal quotient bundle
on the Grassmannian, and $\pi_1$, $\pi_2$ are the projections.

We now apply the inverse image under $f$ and push the complex down to 
$\bP$ in order to obtain
\begin{equation}\label{eq7}
\dotsc\to \wedge^{i+1} Q'\otimes \pi_{2,*}\pi_1^*L^{\otimes j-1}
\to \wedge^{i} Q'\otimes \pi_{2,*}\pi_1^*L^{\otimes j}
\to \wedge^{i-1} Q'\otimes \pi_{2,*}\pi_1^*L^{\otimes j+1} \to\dotsc
\end{equation}
where $Q'$ is the inverse image of $Q$, and the sheaves
$\pi_{2,*}\pi_1^*L^{\otimes j}$ are certain rank $k+1$ vector bundles on $\bP$.

\begin{lemma}\label{lem42}
The direct images of the inverse images of the sections 
of $\pi_2^*\wedge^{i} Q\otimes\pi_1^*L^{\otimes j}$
actually lie in the subsheaves $\wedge^{i} Q'\otimes L_{j}$
of $\wedge^{i} Q'\otimes \pi_{2,*}\pi_1^*L^{\otimes j}$
where
\begin{equation*}
L_j=\Img\Big( H^0(\bP,\pi_{2,*}\pi_1^*L^{\otimes j})\otimes\cO_\bP\to \pi_{2,*}\pi_1^*L^{\otimes j}\Big)=
\begin{cases}
0 & j<0 \\
\cO_\bP & j=0\\
Q' & j=1 \\
\pi_{2,*}\pi_1^*L^{\otimes j} & j\ge 2.
\end{cases}
\end{equation*}
\end{lemma}
\begin{proof}
We noted in (\ref{num26}) that the sections of $\pi_2^*\wedge^i Q\otimes\pi_1^*L^{\otimes j}$ 
on the incidence variety $I_{k,r}$ over the Grassmannian are restrictions of sections of 
$pr_2^*\wedge^i Q\otimes pr_1^*L^{\otimes j}$ on $X\times\Grass_k(V)$. 

This implies that the inverse images of these sections on $I_E$ are restrictions of sections
of $pr_2^*\wedge^i Q'\otimes pr_1^*L^{\otimes j}$ on $X\times \bP$.
hence lie in the image of $H^0(\wedge^i Q')\otimes H^0(L^{\otimes j})$.

Pushing down to $\bP$, we obtain sections of $\wedge^i Q'\otimes \pi_{2,*}\pi_1^*L^{\otimes j}$
which must lie in the subsheaf $\wedge^i Q'\otimes L_j$ as required.
\end{proof}

We need to ensure that the sheaves from Lemma \ref{lem42} form an exact 
complex of sheaves. The next result covers the case that is relevant for our proof in Theorem \ref{thm45}.

\begin{proposition}\label{prop43}
If $k+1$ is invertible in the base field, then the complex
\begin{equation*}
0 \to \wedge^k Q'\otimes \cO
\to \wedge^{k-1} Q'\otimes Q'
\to \wedge^{k-2} Q'\otimes L_2 \to\dotsc
\end{equation*}
on $\bP$ is exact.
\end{proposition}

The proof builds on the next two lemmas.

\begin{lemma1}\label{lem431}
Let $\pi_2^*Q'\to\pi_2^*\pi_{2,*}\pi_1^*L\to \pi_1^*L$
be the restriction of the evaluation map for $\pi_1^*L$ on $I_E$. 
Its direct image under $\pi_{2,*}$ is the restriction of the trace map for $\pi_{2,*}\pi_1^*L$\textup{,}
which we can embed \textup{(}as the middle vertical map\textup{)} into a commutative diagram
\begin{equation*}
\begin{tikzcd}
0 \ar[r] & Q' \ar[r] \ar[d, dashed] & \pi_{2,*} \pi_2^*Q' \ar[r] \ar[d] & Q'\otimes Q'^\vee(-2) \ar[r] \ar[d, dashed] & 0 \\
0 \ar[r] & Q' \ar[r] & \pi_{2,*}\pi_1^*L \ar[r] & \cO(-2) \ar[r] & 0
\end{tikzcd}
\end{equation*}
Both rows of this diagram are split exact. If $k+1$ is invertible in the base field, then the left vertical map is an isomorphism.
\end{lemma1}
\begin{proof}
The split exact rows arise from (\ref{num33}, ii) and (\ref{num34}, iv) in the previous section, and the dashed
arrows exist because the composition $Q'\to\pi_{2,*} \pi_2^*Q'\to \pi_{2,*}\pi_1^*L\to \cO(-2)$ must be zero.

Now assume that the characteristic of the base field does not divide $k+1$.

To understand the left vertical map, we need to analyze the middle vertical map into which it embeds.
Regarding the latter, the extension morphism for $\pi_{2,*}\pi_1^*L$ and the trace map compose to a sheaf homomorphism
\begin{equation*}
\pi_{2,*}\pi_1^*L \to
\pi_{2,*}\pi_2^*(\pi_{2,*}\pi_1^*L)=\pi_{2,*}(\pi_2^*\pi_{2,*}\pi_1^*L)\to
\pi_{2,*}(\pi_1^*L)
\end{equation*}
which agrees with $(k+1)$ times the identity, hence is an isomorphism.

The restriction to the subsheaf $Q'\subset \pi_{2,*}\pi_1^*L$ 
(from the bottom row of our diagram) is thus also an isomorphism, 
and this subsheaf is mapped under the first map into $\pi_{2,*}\pi_2^*Q'$.

The sections of the image of $Q'$ in $\pi_{2,*}\pi_2^*Q'$ correspond to extensions 
of sections from $Q'$, hence the image agrees with the
subsheaf $Q'\subset \pi_{2,*}\pi_2^*Q'$ from the top row of our diagram.

This means that the left vertical map has a section, thus is an isomorphism.
\end{proof}

\begin{lemma1}\label{lem432}
Let $\det (\pi_2^*Q') \to \wedge^{k-1}\pi_2^*Q' \otimes\pi_1^*L$
be the last nontrivial map in the Koszul complex for $\pi_2^*Q'\to\pi_1^*L$ on $I_E$.
After applying $\pi_{2,*}$\textup{,} we can embed this map into a commutative diagram
\begin{equation*}
\begin{tikzcd}[column sep=small]
0 \ar[r] & \det Q'\otimes \cO\ar[r] \ar[d,dashed] & \det Q'\otimes \pi_{2,*}(\cO_I)\ar[r] \ar[d] 
& \det Q'\otimes Q'^\vee(-2) \ar[r] \ar[d,dashed] & 0 \\
0 \ar[r] & \wedge^{k-1}Q'\otimes Q' \ar[r] & \wedge^{k-1}Q'\otimes \pi_{2,*}(\cO_I\otimes pr_1^*L)\ar[r]
& \wedge^{k-1}Q'\otimes \cO(-2)\ar[r] & 0
\end{tikzcd}
\end{equation*}
Both rows of this diagram are split exact. If $k+1$ is invertible in the base field, then the right vertical map is an isomorphism.
\end{lemma1}
\begin{proof}
Up to a twist by a line bundle on $\bP$, this is the dual statement to the previous 
Lemma \ref{lem431}, hence follows by duality. 
Note that $\wedge^{k-1}Q'\cong Q'^\vee\otimes\det Q'$, because $Q'$ is locally free of rank $k$.
\end{proof}

\begin{proof}[Proof of Proposition \textup{\ref{prop43}}]
Consider the following (vertical) short exact sequence of complexes
\begin{equation*}
\begin{tikzcd}[column sep=tiny]
0 \ar[r] & \wedge^k Q'\otimes \cO
\ar[r] \ar[d] & \wedge^{k-1} Q'\otimes Q'
\ar[r] \ar[d] & \wedge^{k-2} Q'\otimes L_2 \ar[r] \ar[d, equals] & \dotsc\\
0 \ar[r] &  \wedge^k Q'\otimes \pi_{2,*}\pi_1^*\cO_X
\ar[r] \ar[d] &  \wedge^{k-1} Q'\otimes \pi_{2,*}\pi_1^*L
\ar[r] \ar[d] &  \wedge^{k-2} Q'\otimes \pi_{2,*}\pi_1^*L^{\otimes 2} \ar[r] & \dotsc\\
& \wedge^k Q'\otimes Q'^\vee(-2) \ar[r] 
& \wedge^{k-1} Q'\otimes \cO(-2)
\end{tikzcd}
\end{equation*}
The two leftmost columns were displayed as rows in the statement of Lemma \ref{lem432}.
The middle row arises from an exact sequence of vector bundles on the Grassmannian 
via pullback, followed by the 
direct image of a finite map. The inverse image preserves exactness, as the complex 
is locally split; and the direct image also preserves exactness, as a finite map. 
The map in the bottom row is an isomorphism by Lemma \ref{lem432}, hence
the result follows from the long exact sequence of homology of a short exact sequence of complexes.
\end{proof}

Finally we investigate the inverse image maps on global sections
\begin{equation*}
\Phi^i_j\colon \wedge^i H^0(X,L)\otimes H^0(X,L^{\otimes j}) \to H^0(\bP,\wedge^i Q'\otimes L_j)
\end{equation*}

\begin{theorem}[Base Change Theorem]\label{thm44}
Suppose that the characteristic of the base field is $0$ or exceeds $i$.
Consider the composition of the inverse image map for sections and the projection to $\bP$,
\begin{equation*}
\Phi_i^j\colon H^0(I_{k,r},\wedge^i \pi_2^*Q\otimes \pi_1^*L^{\otimes j})
\to H^0(\bP,\wedge^i Q'\otimes L_j).
\end{equation*}
as described in Lemma \textup{\ref{lem42}}.
\begin{enumerate}
\item[\textup{1.}] $\Phi_i^j$ is surjective.
\item[\textup{2.}] If $i\le k-2-\sigmax$\textup{,} or if $j=0$ and $i\le k-\sigmax$\textup{,} then 
$\Phi_i^j$ is bijective.
\end{enumerate}
\end{theorem}

The bijectivity of $\Phi_k^0$ in the even genus case ($\sigmax=0$) is one of the key results
in Voisin's proof \cite[3.18]{V1}.
For odd sectional genus, there is a short exact sequence
\begin{equation*}
0\to S^{k-1}H^0(E)\to\wedge^k H^0(L)\to H^0(\bP,\wedge^k Q')\to 0
\end{equation*}
and $\wedge^k Q'=\cO_\bP(k-1)$.

\begin{proof}
First consider $j=0$, $L_0=\cO$: In characteristic $0$,
the exterior powers of $Q'$ can be resolved \cite[3.1]{Le} as\footnote{Lazarsfeld 
pointed out that these resolutions extend to any field where $i!$ is invertible \cite{Te}.}
\begin{equation}\label{eq7a}
0\to F_{i+1}^{(i)}\to \dotsc \to F_1^{(i)}\to F_0^{(i)} \to \wedge^i Q'\to 0
\end{equation}
 derived from (\ref{eq6}) where
\begin{multline*}
F_l^{(i)}=\Big(\wedge^{i-l} H^0(L)\otimes S^l\big(H^0(E)\otimes\cO(-1)\big)\Big)\\
\oplus
\Big(\wedge^{i-l+1} H^0(L)\otimes S^{l-2}\big(H^0(E)\otimes\cO(-1)\big)\otimes\cO(-2)\Big).
\end{multline*}
Here the first term is understood to appear only for $l\le i$, and the second term only for $l\ge 2$.

We note that $F^{(i)}_0=\wedge^i H^0(L)\otimes\cO$, and more generally that 
$F^{(i)}_l$ is a direct sum of copies of $\cO(-l)$, hence the resolution is linear.

Surjectivity of $H^0(\bP,F_0^{(i)})=\wedge^i H^0(L)\to H^0(\bP,\wedge^i Q')$ is immediate \cite[B.1.3]{L}, 
and injectivity follows as long as the last term $F_{i+1}^{(i)}=\oplus\,\cO(-i-1)$ satisfies
$i+1\le\dim\bP=k+1-\sigma$.

Here is an alternative argument based on Castelnuovo-Mumford regularity: As long as the 
characteristic is $0$ or exceeds $i$, the $0$-regularity of $\wedge^i Q'$
follows from \cite[1.8.10]{L}. Hence the map $\wedge^i H^0(Q')=\wedge^i H^0(L)\to H^0(\wedge^i Q')$ is surjective.
To show injectivity, it suffices to verify that both vector spaces have the same dimension. 
As $\wedge^i Q'$ is $0$-regular, the dimension of $H^0(\wedge^i Q')$ agrees with 
the Euler characteristic $\chi(\wedge^iQ')$, 
and the latter can be determined by Riemann-Roch via a calculation with Chern classes.

Now consider $j\ge 1$: $L_1=Q'$ is $0$-regular and has a linear resolution with three terms. The same holds
for $L_j$ ($j\ge 2$), after twisting (\ref{eq5}) by $pr_1^*L^{\otimes j}$ and projecting to $\bP$:
\begin{multline*}
0\to H^0(L^{\otimes(j-1)})\otimes\cO_\bP(-2)\to H^0(E\otimes L^{\otimes(j-1)})\otimes\cO_\bP(-1) \\
\to H^0(L^{\otimes j})\otimes\cO_\bP\to \pi_{2,*}\pi_1^*L^{\otimes j}\to 0.
\end{multline*}

The tensor product of the resolution (\ref{eq7a}) for $\wedge^i Q'$ 
(with last term $\oplus\,\cO(-i-1)$) and the resolution for $L_j$ (for $j\ge 1$) provide a linear 
resolution for $\wedge^i Q'\otimes L_j$ with ending term $\oplus\,\cO(-i-3)$.
This resolution implies that the multiplication map
\begin{equation*}
H^0(\bP,L_j)\otimes H^0(\bP,\wedge^i Q') \to H^0(\bP,L_j\otimes\wedge^i Q')
\end{equation*}
is an isomorphism for $i\le k-2-\sigmax$, and can be used to identify its kernel in the other cases.
\end{proof}

\begin{theorem}\label{thm45}
Let $(X,L)$ be a \emph{K3} surface $X$ with a very ample line bundle $L$ such that $L\cdot L=4k-2$. 
Assume that every effective divisor in the associated linear system $\vert L\vert$ is reduced and irreducible.
Then we have $K_{k-2,2}(X,L)=0$.
\end{theorem}
\begin{proof}
Consider the following commutative diagram with vertical base change maps
\begin{equation*}
\begin{tikzcd}[column sep=small]
0 \ar[r] & \wedge^k H^0(L) \ar[r] \ar[d, "\cong"] &
\wedge^{k-1} H^0(L)\otimes H^0(L) \ar[r] \ar[d, two heads] &
\wedge^{k-2} H^0(L)\otimes H^0(L^{\otimes 2}) \ar[r] \ar[d, "\cong"] & \dotsc\\
0 \ar[r] & H^0(\bP,\wedge^k Q')
\ar[r] & H^0(\bP,\wedge^{k-1} Q'\otimes Q')
\ar[r] & H^0(\bP,\wedge^{k-2} Q'\otimes L_2) \ar[r] & \dotsc
\end{tikzcd}
\end{equation*}
The top row calculates Koszul cohomology, and the bottom row arises from
transferring the geometric Koszul complex to $\bP=\bP(H^0(E)^\vee)$ and taking global sections.

The vertical arrows are surjective by Theorem \ref{thm44}, and we consider the homology of the rows:
\begin{enumerate}
\item[(i)] The homology groups of the top row are $K_{k,0}(X;L)=0$, $K_{k-1,1}(X;L)$ and $K_{k-2,2}(X;L)$.
\item[(ii)] The homology groups of the bottom row vanish: 
The left-exactness of the global section functor implies exactness at the first two terms; exactness
at the third term is ensured by the vanishing of $H^1(\bP,\wedge^k Q')=H^1(\bP,\cO(k))$.
\item[(iii)] The vertical maps on the right and on the left are injective by Theorem \ref{thm44}, and the kernel of the 
central map is isomorphic to $S^{k-2}H^0(E)$.
\end{enumerate}

Viewing the diagram as the two lower rows of a vertical short exact sequence of complexes,
we conclude from the corresponding long exact sequence of homology that 
$K_{k-1,1}(X;L)=S^{k-2}H^0(E)$ and $K_{k-2,2}(X;L)=0$ as desired.
\end{proof}

\begin{remark}\label{rem46}
Let us briefly mention what happens in the case of odd sectional genus:

We still obtain the diagram from the previous proof, and (i) and (ii) continue to hold. The kernel complex is
\begin{equation}\label{eq8}
0\to S^{k-1}H^0(E)\to M \to S^{k-3} H^0(E)\otimes H^0(L)\to 0
\end{equation}
where $M$ is resolved as
\begin{multline*}
0\to S^{k-2}H^0(E)\otimes H^0(E)\to \\
\big(H^0(L)\otimes S^{k-3}H^0(E)\big)\oplus S^{k-1}H^0(E)\oplus \big(S^{k-2}H^0(E)\otimes H^0(E)\big)\to M \to 0
\end{multline*}

$K_{k-2,1}(X;L)$ is isomorphic to the cokernel of the map $M \to S^{k-3} H^0(E)\otimes H^0(L)$,
but it is not a priori clear why this map should be surjective.
\end{remark}

\begin{remark}\label{rem47}
For K3 surfaces of odd sectional genus, the space of sections of a Lazarsfeld Mukai bundle is too small
to capture all the information of the Koszul complex, hence we should consider a larger parameter space. 

Actually, there is a canonical choice for such a space: 
Over K3 surfaces of even sectional genus, $\bP(H^0(E)^\vee)$ can be specified independently of the
vector bundle $E$, as the subscheme of the Hilbert scheme $\Hilb^{k+1}(X)$
of points that are not in general position in the embedding by $\vert L\vert$. 
For even sectional genus, this space corresponds to sections of a unique
Lazarsfeld Mukai bundle $E$.

The corresponding subscheme for K3 surfaces of odd sectional genus can be expressed as a projective fibration
over another K3 surface. The fibers correspond to the spaces of sections of a suitable Lazarsfeld Mukai bundle
on $X$, and the base K3 surface parametrizes bundles with these numerical invariants and
is closely related to $X$. Our analysis in (\ref{rem46}) only captured the information in one of those fibres.

We do not know if our setup can be adapted:
\begin{enumerate}
\item[(i)] A universal bundle may not exist over the mirror surface
(but see the discussion in \cite[section 10.2.2]{H}).
\item[(ii)] One might expect, that the mirror surface is isomorphic to the original K3 surface $X$, but 
this is not at all obvious from general principles \cite[7.35]{BBR}.
\end{enumerate}

Another approach could be to work with an infinitesimal deformation of $E$, i.e., with a
non-reduced structure on $\bP(H^0(E)^\vee)$. 
\end{remark}

\begin{remark}[An alternative approach]
Voisin's approach identifies part of the Koszul complex on $X\times \Hilb^{k+1}(X)$ (more exactly,
the inverse image from (\ref{sect29}.2) indicated earlier) 
and takes a further inverse image via a certain rational map
\begin{equation*}
X\times\bP(H^0(E)^\vee)\to \Hilb^{k+1}(X).
\end{equation*}
This rational map fits into the following larger diagram
\begin{equation*}
\begin{tikzcd}
X\times\bP(H^0(E)^\vee)\ar[r, dashed]  \ar[dr, "span"', dashed]
& \Hilb^{k+1}(X) \ar[d, "can", dashed] \\
& \Grass_{k+1}(V)
\end{tikzcd}
\end{equation*}
where all the maps to the Grassmannian are the natural ones.

The diagonal arrow becomes a morphism after blowing up the incidence variety $I_E$, and 
it should be possible to prove the vanishing of $K_{k,1}(X;L)$ based on an analysis of the resulting inverse image.
\end{remark}

\bigskip
\textit{E-mail}: {\tt jk.rathmann@gmail.com}
\end{document}